\documentclass[11pt,notitlepage,a4paper]{article}

\usepackage{amssymb} 
\usepackage[intlimits]{amsmath}
\usepackage{amsfonts}
\usepackage{amsthm} 

\usepackage{mathptmx}

\usepackage[utf8]{inputenc}

\usepackage{hyperref}
\hypersetup{colorlinks=true,linkcolor=RoyalPurple,citecolor=RoyalPurple}

\usepackage{cite}

\usepackage{graphicx}

\usepackage{geometry}

\usepackage{color}
\usepackage[dvipsnames]{xcolor}

\let\oldsqrt\sqrt
\def\sqrt{\mathpalette\DHLhksqrt}
\def\DHLhksqrt#1#2{%
\setbox0=\hbox{$#1\oldsqrt{#2\,}$}\dimen0=\ht0
\advance\dimen0-0.2\ht0
\setbox2=\hbox{\vrule height\ht0 depth -\dimen0}%
{\box0\lower0.4pt\box2}}

\allowdisplaybreaks

\newcommand{\R}{\mathbb{R}} 
\newcommand{\N}{\mathbb{N}} 

\newcommand{\supp}{\textnormal{supp}} 

\renewcommand{\phi}{\varphi}

\newcommand{\m}{m}

\newcommand{\cE}{{\mathcal E}}
\newcommand{\cF}{{\mathcal F}}

\newcommand{\cH}{{\mathcal H}}

\newcommand{\cL}{{\mathcal L}}

\newcommand{\dn}{D}

\theoremstyle{definition}
\newtheorem{defi}{Definition}[section]
\newtheorem{remark}[defi]{Remark}

\theoremstyle{plain} 
\newtheorem{thm}[defi]{Theorem}

\newtheorem{lemma}[defi]{Lemma}

\theoremstyle{definition}

\numberwithin{equation}{section} 

\title{
On fractional higher-order Dirichlet boundary value problems: between the Laplacian and the bilaplacian.
}
\author{
Alberto Salda\~{n}a\footnote{Instituto de Matemáticas, Universidad Nacional Autónoma de México, Circuito Exterior, Ciudad Universitaria, 04510 Coyoacán, CDMX, Mexico. \emph{ alberto.saldana@im.unam.mx}}
}
\date{}

\begin{document}

\maketitle

\begin{abstract}
The solutions of boundary value problems for the Laplacian and the bilaplacian exhibit very different qualitative behaviors. Particularly, the failure of general maximum principles for the bilaplacian implies that solutions of higher-order problems are less rigid and more complex. One way to better understand this transition is to study the intermediate Dirichlet problem in terms of fractional Laplacians.  This survey aims to be an introduction to this type of problems; in particular, the different pointwise notions for these operators is introduced considering a suitable natural extension of the Dirichlet boundary conditions for the fractional setting.  Solutions are obtained variationally and, in the case of the ball, via explicit kernels. The validity of maximum principles for these intermediate problems is also discussed as well as the limiting behavior of solutions when approaching the Laplacian or the bilaplacian case.
\end{abstract}

\section{Introduction}

Let $U\subset \R^N$ ($N\in\N$) be an open bounded set with smooth boundary and $\beta\in(0,1)$. The (inhomogeneous) Dirichlet boundary value problem for the \emph{Laplacian} $\Delta u = \sum_{i=1}^N \partial_{ii}u$ is to find the unique solution $u\in C^{2+\beta}(U)\cap C(\overline{U})$ of 
\begin{align}\label{lp}
 -\Delta u = f\quad \text{ in }U,\qquad  u=g\quad \text{ on }\partial U,
\end{align}
where $f\in C^\beta(U)$ and $g\in C(\partial B)$ are given data.  Analogously, the Dirichlet boundary value problem for the \emph{bilaplacian} is to find the unique solution $u\in C^{4+\beta}(U)\cap C^1(\overline{U})$ of 
\begin{align}\label{bilp}
 \Delta^2 u = f\quad \text{ in }U,\qquad  u=g\quad \text{ on }\partial U,\qquad  -\partial_\nu u=h\quad \text{ on }\partial U,
\end{align}
where $\nu$ denotes the unit outward-pointing normal vector on $\partial U$, $f\in C^\beta(U)$, $g\in C^1(\partial B)$, and $h\in C(\partial B)$ are given data.  The problems \eqref{lp} and \eqref{bilp} are at the core of the linear theory for elliptic operators, and there is an extensive study of their solutions, primarily focusing on existence, uniqueness, regularity, qualitative properties (positivity, symmetry), and representation formulas, see \cite{GT,GGS10}.  

Regarding existence and uniqueness of solutions to \eqref{lp} and \eqref{bilp}, a very powerful approach is given by \emph{variational methods}, where solutions can be found as (unique) minimizers of some energy functional on a suitable Sobolev space.  Once the existence of a solution is established, one of the most important questions regarding its qualitative properties is the following: 
\begin{align*}
\text{Given \emph{nonnegative} data, is the corresponding solution \emph{nonnegative}?} 
\end{align*}
This is sometimes called a \emph{positivity preserving property} or a maximum principle.  To be more precise, for the Laplacian this amounts to the following question: if $f\geq 0$ and $g\geq 0$, is it true that the solution $u$ of \eqref{lp} is nonnegative?, similarly for the bilaplacian: if $f\geq 0$, $g\geq 0$, and $h\geq0$, is it true that the solution $u$ of \eqref{bilp} is nonnegative? 

For the Laplacian, this positivity preserving property \emph{always} holds, and this is a formidable tool in the study of linear and nonlinear elliptic equations (and systems).  Maximum principles are at the heart of a priori bounds, symmetry characterizations, existence, nonexistence, uniqueness, multiplicity, and regularity results, among others; we refer to \cite{GT,SW12,S16,SW15,S16t,BFDST18,ST17} for a glimpse of the diversity of problems and methods in which maximum principles play an essential role. 

For the bilaplacian the situation is more delicate: maximum principles do \emph{not} hold in general and this is linked to the fact that higher-order problems are \textquotedblleft less rigid\textquotedblright, which produces a larger and more complex set of solutions, see for example \cite{GGS10,BFS16,PT01,AJSmp} and the references therein.  There are, however, some partial results regarding positivity preserving properties in the higher-order setting; for example, if $U$ is a ball (or a small perturbation of the ball), $f\geq 0$, $h\geq0$, and $g\equiv 0$, then $u\geq 0$ \cite{GGS10}.  One of the most important (and long-standing) open problems in the theory of higher-order linear elliptic problems is to understand which domains $U$ allow a positivity preserving property for \eqref{bilp} with $f\geq 0$ and $g\equiv h\equiv 0$ \cite{GGS10}.

\medskip

This paper focuses on the intermediate problem between \eqref{lp} and \eqref{bilp} by considering fractional powers of the Laplacian $(-\Delta)^s$ with $s\in(1,2)$.  This fractional higher-order Dirichlet problem has a very rich structure resulting from the mixture between the Laplacian, the bilaplacian, and new purely nonlocal phenomena.  The understanding of this complex structure is interesting in its own right, but another motivation to study these problems comes from the fact that the fractional setting offers a bridge between the two Dirichlet problems \eqref{lp} and \eqref{bilp}, one which is very well-behaved and rigid, whereas the other is laxer and less constrained by the properties of the given data. Such a bridge has the potential to offer a new perspective and novel insights on the many open problems regarding linear and nonlinear higher-order equations.

Boundary value problems for higher-order fractional powers of the Laplacian have not been studied much in the literature so far. Some known results are the following. 
General regularity results have been proved in \cite{G15:2} (see also the survey \cite{G18} and the references therein), a Poho\v zaev identity and an integration by parts formula is given in \cite{RS15}, 
a comparison between different notions of higher-order fractional operators is done in \cite{MN15}, and spectral results are obtained in \cite{G15}.  A discussion on the pointwise definition of $(-\Delta)^s$ can be found in \cite{AJS17b}, explicit integral representations of solutions in \cite{AJS16b,AJS17a,ADFJS18}, and a study of positivity preserving properties in \cite{AJSmp}. 

Our discussion below is based on the results from \cite{AJSmp,AJS17b,AJS16b,AJS17a} and is guided by the following natural questions. For $s\in(1,2)$:
\begin{itemize}
 \item [(Q1)] What would be a suitable pointwise evaluation for the operator $(-\Delta)^s$?
 \item [(Q2)] Is it possible to study higher-order fractional Dirichlet problems variationally?
 \item [(Q3)] What can be said about positivity preserving properties for $(-\Delta)^s$?
 \item [(Q4)] What is the natural extension of (inhomogeneous) Dirichlet boundary conditions for $(-\Delta)^s$ and how can we find solutions?
 \item [(Q5)] What happens with the solutions of $(-\Delta)^su=f$ in $U$ as $s\to 2$ or as $s\to 1$? Do we recover solutions of \eqref{lp} and \eqref{bilp}?
\end{itemize}

We answer the first two questions in Sections \ref{p:sec} and \ref{var:sec} in a rather general setting.  Section \ref{mp:sec} is devoted to answer (Q3) and for (Q4) and (Q5) we study in detail the case of the ball in Sections \ref{balls:sec} and \ref{asym:sec}. Finally, in Appendix \ref{A} we include a brief discussion on the composition of Green functions.

\medskip

To close this introduction we remark that this survey is focused only in the case $s\in(1,2)$ for simplicity and to fix ideas; however, most of the definitions and results discussed below are available for any $s>1$, see Remark \ref{sb2} for some comments on the case $s>2$.  We refer to \cite{AJSmp,AJS17b,AJS16b,AJS17a,ADFJS18} for more details.

\subsection{Notations}\label{Notation}

In the reminder of the paper we use the following standing notation. We fix $B:=B_1(0)$ and $B_r:=B_r(0)$ for $r>0$.  For $m\in \N_0\cup \{\infty\}$ and $U$ open we write $C^{m,0}(U)$ to denote the space of $m$-times continuously differentiable functions in $U$ and, for $\sigma\in(0,1]$ and $s=m+\sigma$, we write $C^s(U):=C^{m,\sigma}(U)$ to denote the space of functions in $C^{m,0}(U)$ whose derivatives of order $m$ are (locally) $\sigma$-H\"older continuous in $U$ or (locally) Lipschitz continuous in $U$ if $\sigma=1$. We denote by $C^s(\overline{U})$ the set of functions $u\in C^s(U)$ such that 
\begin{align}\label{Hn}
\|u\|_{C^s(U)}:=\sum_{|\alpha|\leq m}\|\partial^{\alpha}u\|_{L^{\infty}(U)}+\sum_{|\alpha|=m}\ \sup_{\substack{x,y\in U \\x\neq y}}\  \frac{|\partial^{\alpha}u(x)-\partial^{\alpha}u(y)|}{|x-y|^{\sigma}}<\infty.
\end{align}
Moreover, for $s\in(0,\infty]$, 
\begin{align*}
C^s_c(U):=\{u\in C^s(\R^N): \supp \ u\subset\subset U\}, \qquad C^s_0(U):=\{u\in C^s(\R^N): u= 0 \text{ on $\R^N\setminus U$}\},
\end{align*}
where $\supp\ u:=\overline{\{ x\in U\;:\; u(x)\neq 0\}}$ is the support of $u$.  We also write $A\subset\subset B$ to denote that $A$ is \emph{compactly contained} in $B$, that is, that $\overline{A}$ is a compact set and $\overline{A}\subset B$.
\medskip

We use $u^+:=u_+:= \max\{u,0\}$ to denote the positive part of $u$. For $\beta\in\R$ we set
\begin{equation*}
\delta(x)^{\beta}:=\left\{\begin{aligned}
&(1-|x|^2)^{\beta},&& \quad \text{if $1-|x|^2>0$,}\\
&0,&&\quad  \text{if $1-|x|^2\leq 0$.}
\end{aligned}\right.
\end{equation*}
If $\beta=1$ we simply write $\delta(x)$.  The fractional Sobolev space $H^s(\R^N)$ is given by
\begin{align*}
H^s(\R^N):=\left\{u\in L^2(\R^N)\;:\; (1+|\xi|^{2})^{\frac{s}{2}}\ \cF (u)\in L^2(\R^N)\right\}, 
\end{align*}
where $\cF$ denotes the Fourier transform and, for $U\subset \R^N$ open, the homogeneous Dirichlet fractional Sobolev space is
\begin{align*}
\cH^{s}_0(U)&:=\{u\in H^{s}(\R^N)\;:\; u\equiv 0\;\text{on $\R^N\setminus U$}\}.
\end{align*}
Furthermore, $H^s(U):=\{u\chi_{U}\::\: u\in H^s(\R^N)\}$, where $\chi_U$ is the characteristic function of $U$, namely, $\chi_U(x)=1$ if $x\in U$ and $\chi_U(x)=0$ if $x\not\in U$.  We frequently use the following normalization constants:  
\begin{align}\label{c}
\omega_N:=2\pi^{\frac{N}{2}}\ \Gamma(\frac{N}{2})^{-1},\qquad 
k_{N,s}&:=\frac{2^{1-2s}}{\omega_N{\Gamma(s)}^{2}},\qquad 
\gamma_{N,\sigma}:=\frac{2}{ \Gamma(\sigma)\,\Gamma(1-\sigma)\omega_N},
\end{align}
where $\Gamma$ denotes the usual \emph{Gamma function}.  Finally, we recall that, in dimension one ($N=1$), the boundary integral is meant in the sense $\int_{\partial B}f(\theta)\ d\theta=f(-1)+f(1)$.

\section{Pointwise evaluations}\label{p:sec}

The pointwise definition of the higher-order fractional Laplacian $(-\Delta)^s$ for $s\in(1,2)$ can be a delicate issue and some of its aspects may seem a bit counterintuitive at first glance.  Here we present three ways to understand this operator pointwisely and discuss some of their advantages and disadvantages.  The first one is a classical definition via the Fourier transform, the second one is based on a composition of operators (similarly as in the definition of bilaplacian), and finally the third one is based on higher-order finite differences. This last pointwise notion is the most general and is the one we use in the rest of the paper.  For smooth functions ($C^\infty_c(\R^N)$, for example), all these evaluations agree; but, as soon as one considers less regular elements, differences\textemdash which are crucial to study boundary value problems\textemdash appear.

We also emphasize that, in the fractional setting, the pointwise definition of the operator is \emph{closely linked} to the type of boundary conditions that is being studied. In this survey we concentrate only on \emph{Dirichlet-type} boundary conditions.  To see how different boundary conditions may require a change in the pointwise notion of $(-\Delta)^s$, we refer to \cite{roberta:neumann,nicola:robin} and the references therein, where Neumann and Robin-type boundary conditions are considered for powers $s\in(0,1)$.

\subsection{Via Fourier transform}

Fractional Laplacians can be seen as a \emph{pseudo-differential} operator, that is, they can be defined via the Fourier transform $\cF$ prescribing the \emph{symbol} of the operator, namely, for $s>0$,
\begin{align}\label{Fdef}
 (-\Delta)^s \varphi(x)=\cF^{-1}(|\cdot|^{2s}\cF(\varphi))(x)\qquad \text{ for all }\varphi\in C^\infty_c(\R^N).
\end{align}
This notion has the advantage of relating the structure and properties of the Fourier transform with the higher-order fractional Laplacian but it is a rather indirect pointwise definition, which makes it difficult to perform some explicit pointwise calculations. 

\subsection{Via a composition of operators}

The bilaplacian operator $\Delta^2$ can be simply defined by iterating the Laplacian, that is, 
\begin{align*}
\Delta^2 u(x) = (-\Delta) (-\Delta) u(x)\qquad \text{ for all } u\in C^4(\R^N).
\end{align*}
Analogously, for $s\in(1,2)$, one can define the higher-order Laplacian $(-\Delta)^s$ as a composition of $(-\Delta)$ and $(-\Delta)^{s-1}$, where $(-\Delta)^{s-1}$ is given by
\begin{align}\label{sigmadef}
 (-\Delta)^{s-1} u(x) :=  e_{N,s}\int_{\R^N}\frac{2u(x)-u(x+y)-u(x-y)}{|y|^{N+2(s-1)}}\;dy
\end{align}
with
\begin{align}\label{edef}
 e_{N,s}=-\frac{4^{s-1}\Gamma(\frac{N}{2}+s-1)}{ \pi^{\frac{N}{2}}\Gamma(1-s)}.
\end{align}
Here $e_{N,s}$ is a suitable normalization constant such that \eqref{Fdef} holds and $u$ is such that the integral \eqref{sigmadef} is finite.  The right-hand side of \eqref{sigmadef} is sometimes called a \emph{hypersingular integral}, because the singularity of the kernel $|y|^{N+2(s-1)}$ at zero is not integrable and requires some local smoothness of $u$ to guarantee integrability, for instance, that $u$ is of class $C^{2(s-1)+\alpha}$ at $x$ for some $\alpha>0$. Moreover, to ensure integrability at infinity, one must impose some growth restrictions; this is usually done by requiring that $u$ belongs to the space $\cL^1_{s-1}$, where
\begin{align}\label{R:space}
 \cL^1_{t}:=\Big\{u\in L^1_{loc}(\R^N)\;:\; \int_{\R^N}\frac{|u(x)|}{1+|x|^{N+2t}}\ dx<\infty \Big\}\qquad \text{ for any }t>0.
\end{align}
Let $U$ be an open set in $\R^N$, then the integral \eqref{sigmadef} is finite for $x\in U$ if $u\in C^{2(s-1)+\alpha}(U)\cap \cL^1_{s-1}$.  Furthermore, one can show that $(-\Delta)^{s-1} u\in C^{2+\alpha}(U)$ if $u\in C^{2s+\alpha}(U)\cap \cL^1_{s-1}$, see \cite[Proposition 2.7]{S05}.  Therefore, we can define, for $u\in C^{2s+\alpha}(U)\cap \cL^1_s$ and $x\in U$,
\begin{align}\label{itdef}
 (-\Delta)^s u(x) =(-\Delta)(-\Delta)^{s-1}u(x).
\end{align}
This pointwise evaluation is very helpful for explicit calculations, since the operator $(-\Delta)^{s-1}$ can be computed in some cases (see Appendix \ref{A} or \cite{D12,DKK15}). However, the evaluation \eqref{itdef} has the following disadvantage: to compute $(-\Delta)^{s-1} u$ one requires the growth restriction $u\in \cL^1_{s-1}$, which is not optimal for $(-\Delta)^s$ (see Theorem \ref{poisson:thm} below).

\medskip

We emphasize that the order of the operators in  \eqref{itdef} is very important and it \emph{cannot} be freely interchanged in general, namely, it is \emph{not} true that $(-\Delta)(-\Delta)^{s-1}u(x)$ equals $(-\Delta)^{s-1}(-\Delta)u(x)$; the equality holds only for smooth enough functions, which is not the case in general for solutions of Dirichlet boundary value problems, see Appendix \ref{A} for an explicit computation in this regard. Finally, we mention that other compositions such as $(-\Delta)^r(-\Delta)^t u$ with $r,t\in(0,1)$ and $r+t=s$ are not well suited for the study of boundary value problems; the reason\textemdash similarly as in the case in Appendix \ref{A}\textemdash is that, although for $u\in C^\infty_c(\R^N)$ all these pointwise notions are equivalent, solutions of boundary value problem are not regular enough to guarantee that these compositions are always well defined (note that $(-\Delta)^{s-1}(-\Delta)u(x)$ requires that $u$ is twice weakly differentiable in $\R^N$, because $(-\Delta)^{s-1}$ is a nonlocal operator).

\subsection{Via finite differences}

We now introduce the most general pointwise evaluation of the higher-order fractional Laplacian, which, similarly as in \eqref{sigmadef}, is in terms of hypersingular integrals but involves higher-order finite differences. For $s\in (1,2)$, $U\subset \R^N$ open, $\beta\in(0,1)$, $u\in \cL^1_{s}\cap C^{2s+\beta}(U)$, and $x\in U$, let
 \begin{align}\label{Ds}
(-\Delta)^s u(x):=c_{N,s}\int_{\R^N}\frac{u(x+2y)-4u(x+y)+6u(x)-4u(x-y)+u(x-2y)}{|y|^{N+2s}}\;dy,
 \end{align}
where
 \begin{align}\label{cNms:def}
c_{N,s}=\frac{
\Gamma(\frac{N}{2}+s)
}{
\pi^{\frac{N}{2}}\Gamma(-s)(1-4^{1-s})
}
\end{align}
is a normalization constant such that \eqref{Fdef} holds (see \cite[Theorem 1.9]{AJS17b} for the details).  In the following, whenever we write $(-\Delta)^s u(x)$ as a pointwise evaluation, we always mean it in the sense of \eqref{Ds}.

Explicit pointwise calculations using \eqref{Ds} are slightly more involved than those for \eqref{itdef}, and typically require some combinatorial identities, see \cite{AJS17b}. We remark that, if $u\in\cL^1_{s-1}$, then \eqref{Ds} is equivalent to \eqref{itdef} (note that $\cL^1_{s-1}\subset \cL^1_s$), we state this result next. 

\begin{lemma}[Particular case of Corollary 1.4 in \cite{AJS17b}]\label{eq:lem}
Let $\beta\in(0,1)$, $s\in(1,2)$, $U\subset\R^N$ be smooth open domain, and $u\in C^{2s+\beta}(U)\cap \cL^1_{s-1}$, then, for $x\in U$,
\begin{align*}
 (-\Delta)(-\Delta)^{s-1}u(x)&=e_{N,s}(-\Delta)\int_{\R^N}\frac{2u(x)-u(x+y)-u(x-y)}{|y|^{N+2s}}\;dy\\
 &=c_{N,s}\int_{\R^N}\frac{u(x+2y)-4u(x+y)+6u(x)-4u(x-y)+u(x-2y)}{|y|^{N+2s}}\;dy;
\end{align*}
in particular, the pointwise evaluations \eqref{Ds} and \eqref{itdef} are equivalent for $u\in C^{2s+\beta}(U)\cap \cL^1_{s-1}$.
\end{lemma}

\section{Variational framework}\label{var:sec}

The variational study of the higher-order fractional Laplacian with homogeneous Dirichlet boundary conditions can be framed in suitable fractional Sobolev spaces. To be precise, let $s\in(1,2)$ and recall the definition of the usual fractional Sobolev space 
\begin{align*}
H^s(\R^N):=\left\{u\in L^2(\R^N)\;:\; (1+|\xi|^{2})^{\frac{s}{2}}\ \cF (u)\in L^2(\R^N)\right\}, 
\end{align*}
where $\cF$ denotes the Fourier transform and, for $U\subset \R^N$ open, define the homogeneous Dirichlet fractional Sobolev space
\begin{align}\label{Hs0:def}
\cH^{s}_0(U)&:=\{u\in H^{s}(\R^N)\;:\; u\equiv 0\;\text{on $\R^N\setminus U$}\}
\end{align}
equipped with the norm 
$\|u\|_{\cH^s_0(U)}:=(\|u\|_{L^2(U)}^2+\sum_{i=1}^N\|\partial_i u\|_{L^2(U)}^2+\cE_{s}(u,u))^{\frac{1}{2}}$, 
where $\cE_s$ is a suitable scalar product in $\cH^s_0(U)$. Similarly as in the previous section, we can have three formulas for this scalar product, each one naturally associated to each pointwise evaluation, however, since $\cH^s_0(U)\subset L^2(\R^N)\subset \cL^1_{s-1}\subset \cL^1_s$, these three expressions are \emph{equivalent} for functions in $\cH^s_0(U)$.  Nevertheless, sometimes one expression can be better suited than the other, depending on the object of study.  The three formulas\textemdash via Fourier transform, composition of operators, and finite differences respectively\textemdash are the following: For $u,v\in \cH^s_0(U)$ and $s\in(1,2)$, let
\begin{align}
\cE_{s}(u,v)&=\int_{\R^N} |\xi|^{2s}\cF u(\xi)\cF v(\xi)\ d\xi,\\
&=\frac{e_{N,s}}{2}\int_{\R^N}\int_{\R^N}\frac{(\nabla u(x)-\nabla u(y))\cdot(\nabla v(x)-\nabla v(y))}{|x-y|^{N+2(s-1)}}\ dx\ dy,\\
&=\frac{c_{N,s}}{2}\int_{\R^N}\int_{\R^N}\frac{
(2u(x)-u(x+y)-u(x-y))(2v(x)-v(x+y)-v(x-y))
}{|y|^{N+2s}}\ dxdy,
\end{align}
where the normalization constants $e_{N,s}$ and $c_{N,s}$ are given in \eqref{edef} and \eqref{cNms:def}.  The proof of the equivalence between these expressions can be found in \cite[Theorem 1.8]{AJS17b}.

\medskip

Now, let $f\in L^2(\Omega)$, we say that a function $u\in \cH_0^{s}(\Omega)$ is a \emph{weak solution} of 
\begin{align}\label{wsol:def1}
    (-\Delta)^{s}u&= f\quad\text{ in $\Omega$,}\qquad u=0 \quad\text{ on $\R^N\setminus \Omega$},
\end{align}
if 
\begin{align}\label{wsol:def}
\cE_{s}(u,\varphi)= \int_{\Omega}f(x)\varphi(x)\ dx\qquad \text{for all $\varphi\in \cH^s_0(\Omega)$.}
\end{align}

In this setting we can use Riesz theorem to yield the following existence result.

\begin{thm}[Corollary 3.6 in \cite{AJSmp}]\label{riesz:thm}
Let $U\subset \R^N$ be an open bounded set. Then for any $f\in L^2(U)$ there is a unique weak solution $u\in \cH_0^{s}(U)$ of $(-\Delta)^{s}u=f$ in $U$. 
\end{thm}

The scalar product $\cE_s$ also satisfies the following integration-by-parts-type formula (see \cite[Lemma 2.4]{AJS17a} and \cite[Theorem 1.8]{AJS17b}). 

\begin{lemma}\label{ibyp}
Let $U\subset \R^N$ be an open bounded set with Lipschitz boundary, $\alpha\in(0,1)$, $s\in(1,2)$, $u\in C^{2s+\alpha}(U)\cap\cL^1_{s-1}\cap \cH_0^s(U)$. Then
\begin{align*}
\cE_s(u,\varphi)=\int_{\R^N} u\, (-\Delta)^s \varphi \ dx = \int_{\R^N} \varphi\,(-\Delta)^{s} u\ dx \qquad \text{for all $\varphi\in C^\infty_c(U)$.}
\end{align*}
\end{lemma}

About the regularity of weak solutions, the following is known.

\begin{lemma}[Theorem 2.2 in \cite{G15:3}]
Let $U\subset \R^N$ be a bounded smooth domain and $\beta\in(0,1)$ such that $2s+\beta\not\in\N$. If $f\in C^\beta(\overline{U})$ and $u\in \cH_0^{s}(U)$ is a weak solution of $(-\Delta)^{s}u=f$ in $U$, then $u\in C^s_0(U)\cap C^{2s+\beta}(U)$.
\end{lemma}

\section{Positivity preserving properties}\label{mp:sec}

As mentioned in the introduction, the Laplacian possesses the following well-known general maximum principle. Let $U$ be a bounded domain with smooth boundary in $\R^N$ and let $H^1_0(U)$ denote the usual Sobolev space of weakly differentiable functions with zero trace at $\partial U$.
\begin{lemma}\label{mp:lem}
If $f\in L^2(U)$ is nonnegative in $U$ and $u\in H^1_0(U)$ is a weak solution of
\begin{align*}
(-\Delta)u = f\quad \text{ in }U,\qquad u=0\quad \text{ on }\partial U, 
\end{align*}
that is,
\begin{align}\label{utest}
 \int_U\nabla u \nabla \phi\ dx = \int_U \phi f\ dx\qquad \text{ for all }\phi\in H^1_0(U),
 \end{align}
then $u\geq 0$ a.e. in $U$.
 \end{lemma}
\begin{proof}
 Since $u^-:=\min\{0,u\} \in H^1_0(U)$, then \eqref{utest} implies that 
 \begin{align*}
0\leq\|u^-\|^2_{H^1_0(U)}=\int_U|\nabla u^-|^2\ dx=\int_U\nabla u \nabla u^-\ dx = \int_U u^- f\ dx\leq 0,  
 \end{align*}
 that is, $\|u^-\|_{H^1_0(U)}=0$ and therefore $u^- \equiv 0$ a.e. in $U$.
\end{proof}

A very similar proof can be done to show the validity of maximum principles for the fractional Laplacian $(-\Delta)^s$ for $s\in(0,1)$, see \cite{AJSmp}.  Observe that there are two important ingredients in the proof of Lemma \ref{mp:lem}: the variational characterization of the solution \eqref{utest} and the belonging of the negative part $u^-$ to the test space $H^1_0(U)$.  Since the gradient $\nabla u^-$ has a jump discontinuity at the level set $\{x\in U \::\:u(x)=0\}$, we have that $u^-$ is not twice weakly differentiable in general, and therefore $u^-\not\in H^2(U)$, which prevents that a similar proof can be performed for the bilaplacian\footnote{Here $H^2(U)$ denotes the Sobolev space of functions which are twice weakly differentiable in $U$ and we say that $u\in H^2(U)\cap H^1_0(U)$ is a weak solutions of $(\Delta)^2u = f$ in $U$ and $\partial_\nu u=u=0$ on $\partial U$ if $\int_U \Delta u\Delta \phi\ dx = \int_U \phi f\ dx$ for all $\phi\in H^2(U)\cap H^1_0(U)$.}.

Interestingly, in \cite[Th\'{e}or\`{e}me 1]{M89} it is shown that
\begin{align*}
u^-\in H^s(U)\qquad \text{ if }u\in H^s(U)\text{ and }s\in\Big(\,0\,,\,\frac{3}{2}\,\Big),
\end{align*}
and, as explained in the previous section, the problem $(-\Delta)^s u = f$ has a variational structure. Since these are the main ingredients in the proof of maximum principles for $s=1$\textemdash which uses $u^-$ as a test function\textemdash it was conjectured that maximum principles would hold for the higher-order fractional Laplacian if $s\in(0,\frac{3}{2})$. However, our next result reveals that the positivity preserving property fails to hold in general for $s\in(1,2)$, therefore it is \emph{not} the belonging of $u^-$ to the space of test functions the reason why maximum principles hold for $s=1$.

\begin{thm}\label{main:thm:point}[Particular case of Theorem 1.1 in \cite{AJSmp}]
Let $N\in\N$, $s\in(1,2)$, $U\subset\R^N$ be an open bounded smooth domain, let $B$ be an open ball compactly contained in $\R^N\setminus U$, and let $\Omega:= U\cup B$. There are $f\in C^\infty(\overline{\Omega})$ and a sign-changing $u\in C^s(\R^N)\cap C^{\infty}(\Omega)\cap L^\infty(\R^N)\cap \cH^s_0(\Omega)$ such that
\begin{align*}
 (-\Delta)^s u=f>0\quad \text{ in }\Omega,\qquad u=0\quad\text{ on }\R^N\backslash \Omega,\qquad u\lneq 0\quad \text{ in }U,\qquad \text{ and }u> 0\quad \text{ in }B.
\end{align*}
\end{thm}

\begin{figure}[h!]
\begin{center}
\includegraphics[height=3cm]{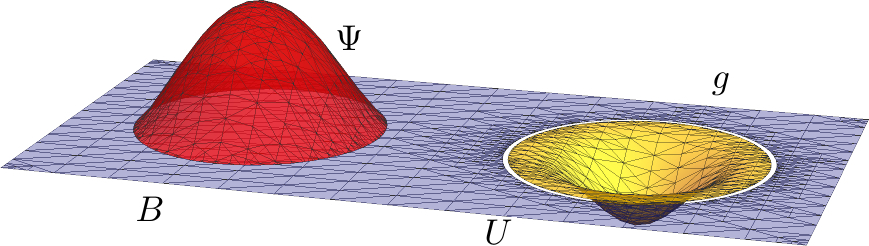}
\end{center}
\caption{Theorem \ref{main:thm:point} is shown via an explicit counterexample $u=\psi+g$ with the above shape.}
\end{figure}

In fact, one can show that the Green function $G_s^\Omega$ associated to two disjoint balls is positive if $s\in(0,1)$ but sign changing if $s\in(1,2)$. 

\begin{thm}[Particular case of Theorem 1.10 in \cite{AJS17a}]\label{even:thm2}
Let $N\in \mathbb N$, $e_1=(1,0,\ldots,0)\in \R^N$, $B=B_1(0)$, $V=B_1(3e_1)$ and $\Omega=B\cup V$. Then 
\begin{align*}
G_s^\Omega&>0\qquad \text{ in }\quad \{(x,y)\in(B\times B)\cup (V\times V)\::\: x \neq y\},\\
G_s^\Omega&>0\qquad \text{ in }\quad (B\times V)\cup (V\times B),\qquad \text{ if }\quad s\in(0,1),\\
G_s^\Omega&<0\qquad \text{ in }\quad (B\times V)\cup (V\times B),\qquad \text{ if }\quad s\in(1,2).
\end{align*}
\end{thm}

\begin{figure}[h!]
\begin{center}
\includegraphics[height=3.8cm]{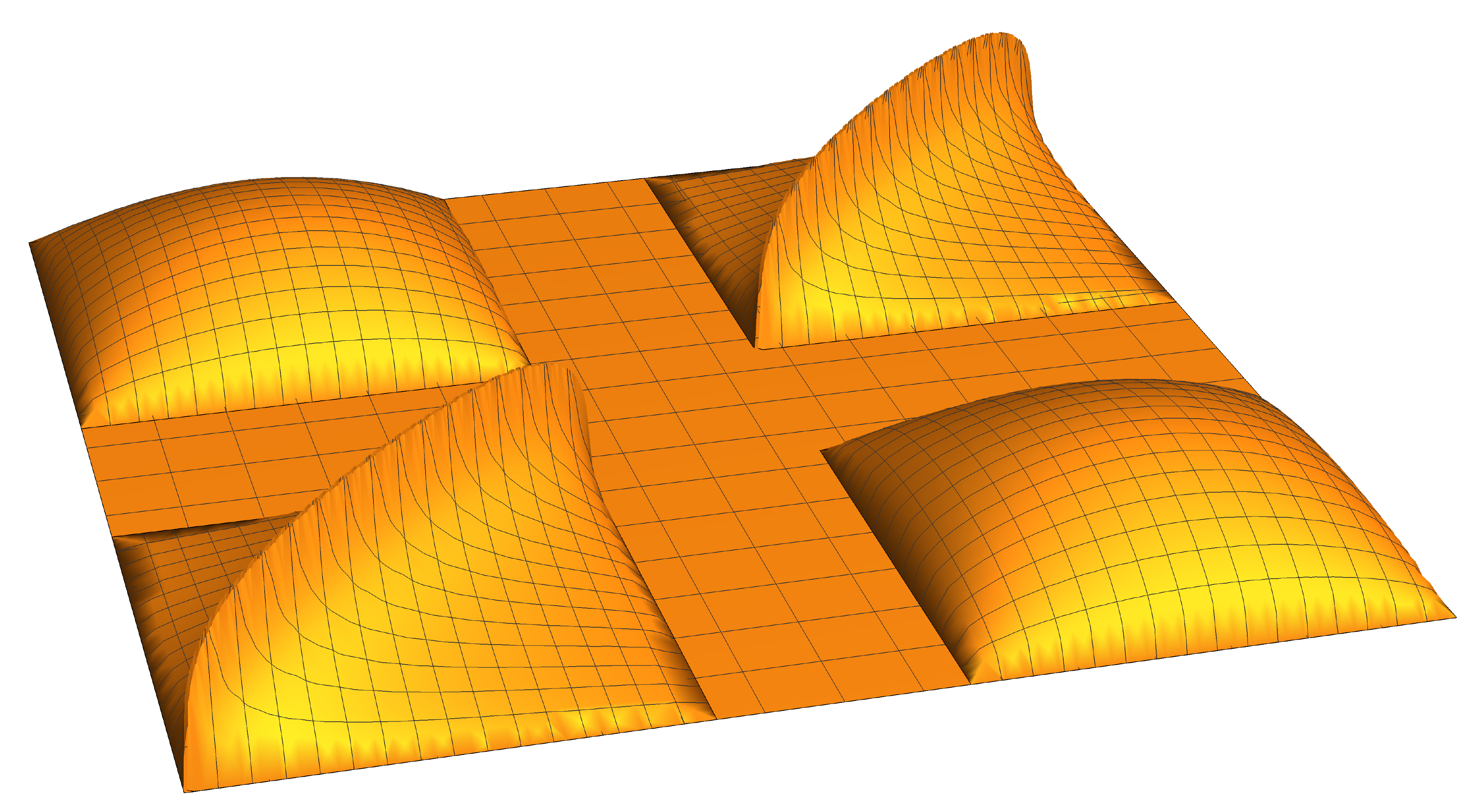}\ \includegraphics[height=3.4cm]{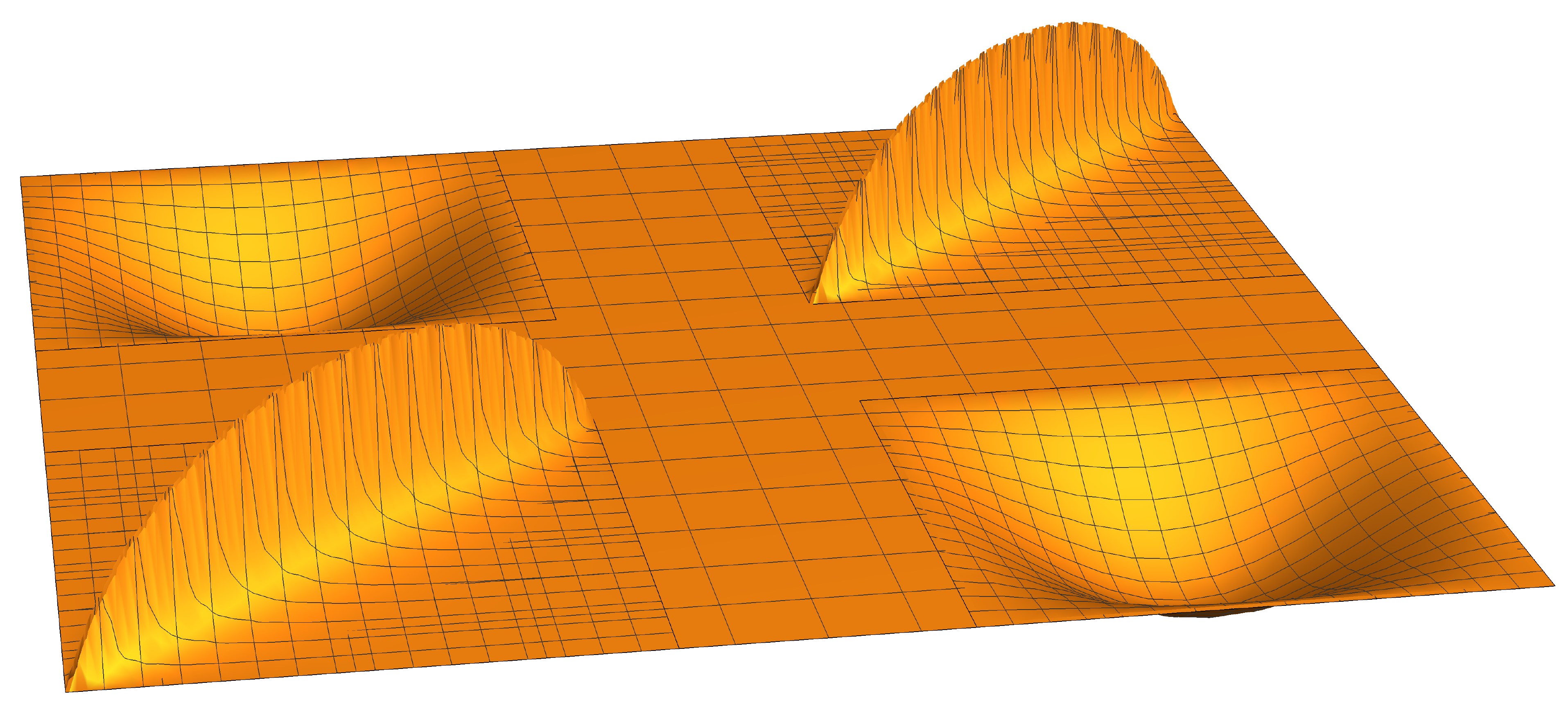}
\end{center}
\caption{The Green function $G_s^\Omega$ for $(-\Delta)^s$ in $\Omega=(-1,1)\cup(2,4)$, on the left for $s\in(0,1)$ and on the right for $s\in(1,2)$.}
\end{figure}

The domain considered in Theorem \ref{main:thm:point} is disconnected, however one can use a perturbation argument to join the domains with a thin tube and find a sign-changing solution in a \emph{connected domain}\footnote{A slight variation of this argument can be used to find sign-changing solutions also in connected \emph{smooth} domains}.

\begin{thm}[Theorem 1.11 in \cite{AJS17a}]\label{connected:cor}
	Let $N\geq 2$, $s\in(1,2)$, $\Omega=B_1(0)\cup B_1(3e_1)$, $L:=\{t e_1\::\: 0<t<3\}$, and 
	\begin{align}\label{omegabn}
	\Omega_n=\Omega \cup \{\ x\in\R^N\::\: \operatorname{dist}(x\ ,\ L)<\frac{1}{n}\ \}\qquad \text{ for } n\in\N.
	\end{align}
	There is $n\in\N$, a nonnegative function $f_n\in L^\infty(\Omega_n)$, and a weak solution $u_n\in\cH_0^{s}(\Omega_n)$ of $(-\Delta)^{s}u_n=f_n\geq 0$ in $\Omega_n$, $u_n=0$ on $\R^N\backslash\overline{\Omega_n}$,
	such that $\operatorname{essinf}_{\Omega_n} u_n<0$ and $\operatorname{esssup}_{\Omega_n} u_n>0$.
\end{thm}

\begin{figure}[h!]
\begin{center}
\includegraphics[height=3cm]{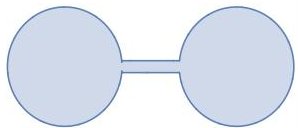}
\end{center}
\caption{Example of the domain $\Omega_n$ given in \eqref{omegabn} for $N=2$.}
\end{figure}

Although maximum principles do not hold in general domains in the higher-order fractional case, we shall see in the next section that positivity preserving properties \emph{do} hold whenever the domain is a ball.

\section{Explicit formulas for the inhomogeneous Dirichlet problem in balls}\label{balls:sec}

In the reminder of the paper we focus our attention on the case of a unitary ball 
\begin{align*}
B:=\{x\in\R^N:|x|<1\}\qquad \text{ in }\R^N,\ N\in \N.  
\end{align*}

The purpose of this section is to explore the notion of the inhomogeneous Dirichlet problem for the higher-order fractional Laplacian, namely, we study what is the natural extension of Dirichlet boundary conditions in order to have well-posed problems for $(-\Delta)^s$ and such that one recovers \eqref{lp} and \eqref{bilp} in the limit as $s\to 1$ or $s\to 2$.

\medskip

A first remark is that the operator $(-\Delta)^s$ is \emph{nonlocal}, since the pointwise computation of $(-\Delta)^s u(x)$ involves all the values of $u$ in $\R^N$, see \eqref{Ds}. Therefore, a natural \textquotedblleft boundary condition\textquotedblright\ is to prescribe values in the complement of the domain, that is,
\begin{align*}
 u=v\qquad \text{ on }\mathbb R^N\backslash\overline{B},
\end{align*}
for some suitable $v\in\cL^1_s$.  We shall see below that the behavior of $v$ close to the boundary $\partial B$ is particularly important when trying to find an explicit representation formula for the solution. 

\medskip

Furthermore, one can also prescribe data at the boundary $\partial B$ using appropriate \emph{traces}. To be precise, for $\beta\in\R$ let
\begin{equation*}
\delta(x)^{\beta}:=\left\{\begin{aligned}
&(1-|x|^2)^{\beta},&& \quad \text{if $1-|x|^2>0$,}\\
&0,&&\quad  \text{if $1-|x|^2\leq 0$.}
\end{aligned}\right.
\end{equation*}
Then, for suitable $u$ and for all $z\in \partial B$, let
\begin{equation}\label{t:def}
\begin{aligned}
 D^{s-2}u(z)&:=\lim_{\substack {x\to z\\ x\in B}}\delta(x)^{2-s}\Big(u(x)
 -\lim_{\substack {y\to z\\ y\in \R^N\backslash B}}u(y)\Big),\\
 D^{s-1}u(z)&:=-\lim_{\substack {x\to z\\ x\in B}}\frac{\partial}{\partial|x|} \Big[\delta(x)^{2-s}\Big(u(x)
 -\lim_{\substack {y\to z\\ y\in \R^N\backslash B}}u(y)\Big)\Big],
\end{aligned}
\end{equation}
where all the limits are always meant in the normal direction, that is, with $\frac{x}{|x|}=\frac{y}{|y|}=z$.  At first glance, the traces \eqref{t:def} might look strange and cumbersome, but let us analyze closely these definitions.  A first remark is that, if $u=0$ in $\R^N\backslash B$, then the traces reduce to 
\begin{align}\label{red}
 D^{s-2}u(z)=\lim_{\substack {x\to z\\ x\in B}}\delta(x)^{2-s}u(x),\qquad 
 D^{s-1}u(z)=-\lim_{\substack {x\to z\\ x\in B}}\frac{\partial}{\partial|x|}[\delta(x)^{2-s} u(x)].
\end{align}
For $s=2$ the weight $\delta^{2-s}$ disappears and \eqref{red} are exactly the Dirichlet boundary conditions for the bilaplacian ($D^0u=u$ and $D^1u=-\partial_\nu u$ on $\partial B$).  For $s\in (1,2)$ note that $2-s\in(0,1)$ and therefore, if $u\in C(B)$ and $D^{s-2}u(z)=\lim_{x\to z}\delta(x)^{2-s}u(x)\neq 0$ for some $z\in\partial B$, then $u$ must be singular at $z$ (see Figure \ref{5f} below).  Solutions satisfying these kind of boundary conditions are sometimes called \emph{very large solutions} and they have been studied in \cite{nicola} for the case $s\in(0,1)$ using a similar trace operator. We also mention that trace operators combining weights and derivatives as in \eqref{red} were also used in \cite[Theorem 6.1]{G15:2} (see also \cite{G15:3}), to study solvability of pseudodifferential operators in a more general setting. 

\begin{remark}\label{rem}
The traces in \eqref{t:def} are different from the ones used in \cite{AJS17a}, which are given by
\begin{equation}\label{other}
\begin{aligned}
 \widetilde D^{s-2}u(z)&=\lim_{\substack {x\to z\\ x\in B}}[\delta(x)^{2-s} u(x)],\\
 \widetilde D^{s-1}u(z)&=-\lim_{\substack {x\to z\\ x\in B}}\frac{\partial}{\partial|x|^2}[\delta(x)^{2-s} u(x)].
\end{aligned}
\end{equation}
The two main differences between \eqref{t:def} and \eqref{other} are the use of the differential operator $\frac{\partial}{\partial |x|}$ instead of $\frac{\partial}{\partial |x|^2}$ and the limit $\lim\limits_{\substack {y\to z\\ y\in \R^N\backslash B}}u(y)$, which does not appear in \eqref{other}.  The reason for the first change is that the operator $\frac{\partial}{\partial |x|}$ is simply the normal derivative $\partial_\nu$ at $\partial B$, which is more common in the study of boundary value problems and substituting $\frac{\partial}{\partial |x|}$ instead of $\frac{\partial}{\partial |x|^2}$ does not imply many changes for $s\in(1,2)$, in fact, 
\begin{align*}
\frac{\partial}{\partial |x|^2}f(x)=\frac{1}{2}\frac{\partial}{\partial |x|}f(x)\qquad \text{ at }\partial B. 
\end{align*}
 The second change, that is, the limit from outside the ball $\R^N\backslash B$, is necessary to consider more general data in $\R^N\backslash B$ which may not vanish close to $\partial B$, see Theorem~\ref{xtra:thm} below.
\end{remark}

In the next subsections we show how to construct solutions to the fractional higher-order inhomogeneous Dirichlet boundary value problem via explicit kernels. 

\subsection{The Green function}

For $x,y\in \R^N,$ $x\neq y$ let $\rho(x,y)=\delta(x)\delta(y)|x-y|^{-2}$ and define
\begin{equation}\label{green}
G_s(x,y)\ =\ k_{N,s} {|x-y|}^{2s-N}\int_0^{\rho(x,y)}\frac{t^{s-1}}{(t+1)^{\frac{N}{2}}}\ dt,
\qquad s>0,\ \ \ x,y\in \R^N,\ \ \ x\neq y,
\end{equation}
where $k_{N,s}$ is a positive normalization constant given in \eqref{c}. The kernel \eqref{green} is known as Boggio's formula (see \cite{B1905,BGR61,GGS10,DG2016,AJS16b}) or Green function for $(-\Delta)^s$ in $B$.  Using this kernel we state the following existence and uniqueness result. Observe also that, since $G_s$ is a positive kernel, a maximum principle is automatically satisfied.

\begin{thm}[Theorem 1.1 in \cite{AJS16b} and Theorem 1.4 in \cite{AJS17a}]\label{green:thm}
Let $s\in(1,2)$, $N\in \N$, $f\in C^{\alpha}(\overline{B})$ for some $\alpha\in(0,1)$ such that $2s+\alpha\not\in\N$, and
\begin{align}\label{Gsu}
u:\R^N\to\R\quad \text{ be given by}\quad u(x)&:=\ \int_{B} G_{s}(x,y)\,f(y)\ dy,
\end{align}
then $u\in C^{2s+\alpha}(B)\cap C_0^s(B)$ is the unique pointwise solution (in $\cH^s_0(B)$) of
\begin{align}\label{mp}
 (-\Delta)^s u = f\quad \text{ in }B,\qquad  u\equiv 0\quad \text{ on }\mathbb R^N\backslash B,
 \qquad D^{s-2}u=D^{s-1}u=0\quad \text{ on }\partial B,
\end{align}
and there is $C>0$ such that $\|\operatorname{dist}(\cdot,\partial B)^{-s}u\|_{L^{\infty}(B)}<C\|f\|_{L^{\infty}(B)}. $ Furthermore, since $G_s$ is a positive kernel, if $f\geq 0$ in $U$ and $f\not\equiv 0$, then $u>0$ in $U$.
\end{thm}

\begin{figure}[h!]
\begin{center}
\includegraphics[height=4.5cm]{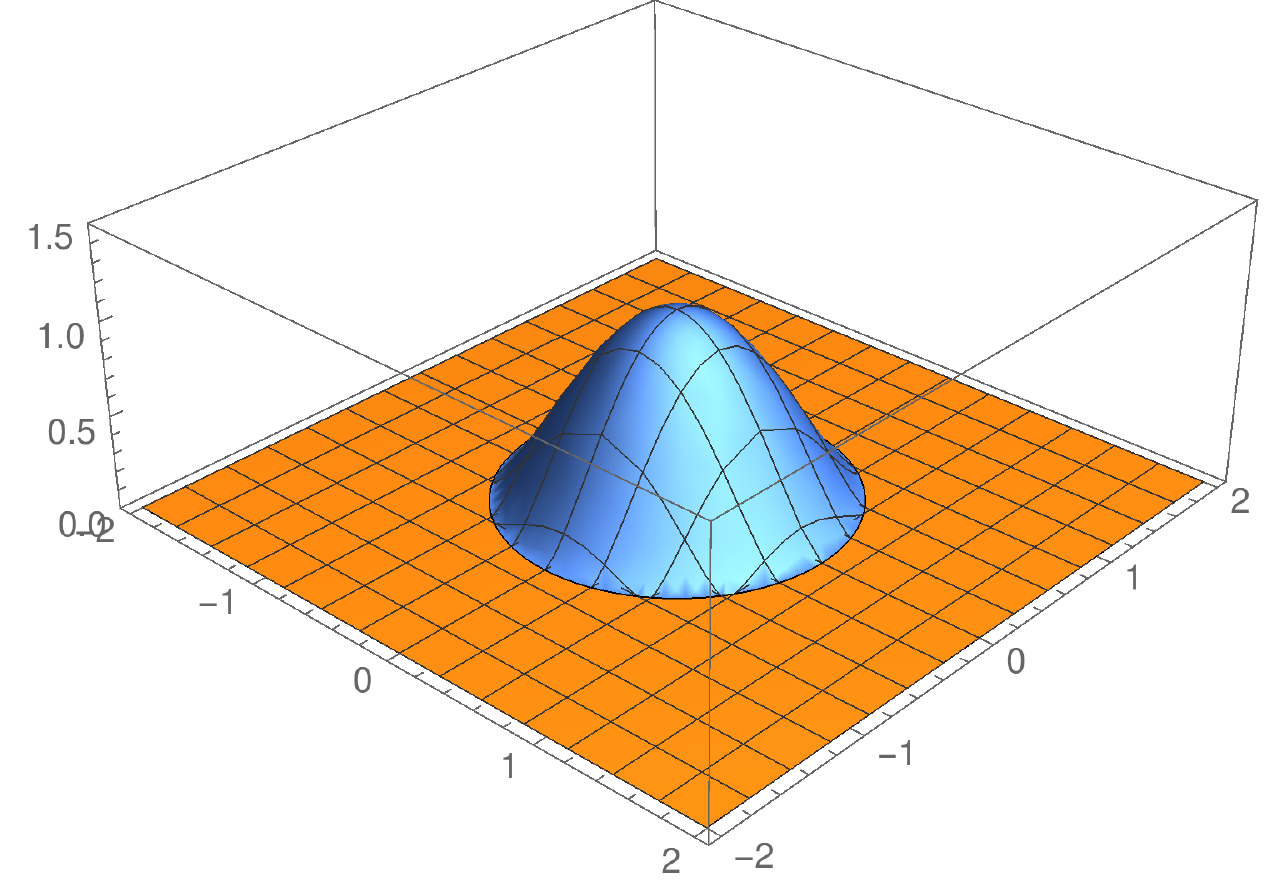}
\end{center}
\caption{Solution of \eqref{mp} with $f\equiv 1$.}\label{3f}
\end{figure}

This is a good moment to make an interesting remark. If $g=h=0$ on $\partial B$, then the function $u$ given by \eqref{Gsu} is the solution of \eqref{bilp} if $s=2$ and of \eqref{lp} if $s=1$. We can also consider the function arising from using twice the kernel $G_1$, namely,
\begin{align*}
 \widetilde v(x):=\int_{B}G_{1}(x,z)\int_{B}G_{1}(z,y)f(y)\ dy\ dz.
\end{align*}
Then $\widetilde v$ is a solution of $\Delta^2 \widetilde v = f$ in $B$ and $\widetilde v$ satisfies \emph{Navier boundary conditions}, that is, $\widetilde v=\Delta \widetilde v =0$ on $\partial B$. Observe that $\widetilde v$ is not a solution of \eqref{bilp} (for example, if $f>0$ in $B$ then $-\partial_\nu \widetilde v >0$ on $\partial B$, by Hopf Lemma). In the fractional case $s\in (1,2)$, one can also consider the function $u$ given by \eqref{Gsu} with $f\equiv 1$ in $B$ and compare it with $v,w:\R^N\to\R$ given by
\begin{align*}
 v(x):=\ \int_{B}G_{s-1}(x,z)\int_{B}G_{1}(z,y)\ dy\ dz,\qquad w(x):=\ \int_{B}G_1(x,z)\int_{B}G_{s-1}(z,y)\ dy\ dz.
\end{align*}
We show in Appendix \ref{A} that $u$, $v$, and $w$ solve different boundary conditions.  Furthermore, we also show the following surprising fact:
\begin{align*}
\text{$(-\Delta)^s u=(-\Delta)^s v=1$ in $B$, but $(-\Delta)^s w\neq 1$ in $B$.}
\end{align*}
There are several reasons for this somewhat unexpected behavior. A particularly important factor is the regularity of these functions at the boundary $\partial B$. We refer to the explicit calculations in Appendix \ref{A} for the details.

\subsection{The boundary Poisson kernels}\label{poisson:subsec}

For $x\in\R^N$ and $z\in\partial B$ with $x\neq z$ let
\begin{equation}\label{Eden}
\begin{aligned}
 E_{s-1}(x,z)&:=\frac{1}{2\omega_N}\frac{\delta(x)^s}{|x-z|^N},\\
 E_{s-2}(x,z)&:=\frac{1}{4\omega_N}\frac{\delta(x)^{s}}{|x-z|^{N+2}}(N\delta(x)-(N-4)|x-z|^2),
\end{aligned}
\end{equation}
where $\omega_N$ is a normalization constant given in \eqref{c}.  The kernels \eqref{Eden} are called the \emph{boundary Poisson kernels} for $(-\Delta)^s$ in $B$ or \emph{Edenhofer kernels} in honor of Johann Edenhofer who first state their formula in the case $s\in\N$, see \cite{E75-2}.  Using these kernels we can now solve the following boundary value problems.

\begin{thm}[Theorem 1.4 in \cite{AJS17a}]
Let $g\in C^{1,0}(\partial B)$ and $u:\R^N\to\R$ be given by 
 \begin{align*}
 u(x)=\int_{\partial B}E_{s-2}(x,\theta)\ g(\theta)\ d\theta\qquad \text{for $x\in \R^N$}.
 \end{align*}
 Then, $u\in C^{\infty}(B)$, $\delta^{2-s}u\in C^{1,0}(\overline{B})$, and
 \begin{align}\label{p:ed}
 (-\Delta)^s u=0\ \ \text{ in }\ \ B,\quad u=0\ \ \text{ on } \mathbb R^N\backslash\overline{B},\quad \dn^{s-2} u= g\ \ \text{ on } \partial B,\quad \dn^{s-1} u= 0\ \ \text{ on } \partial B.
 \end{align}
\end{thm}

\begin{figure}[h!]
\begin{center}
\includegraphics[height=5cm]{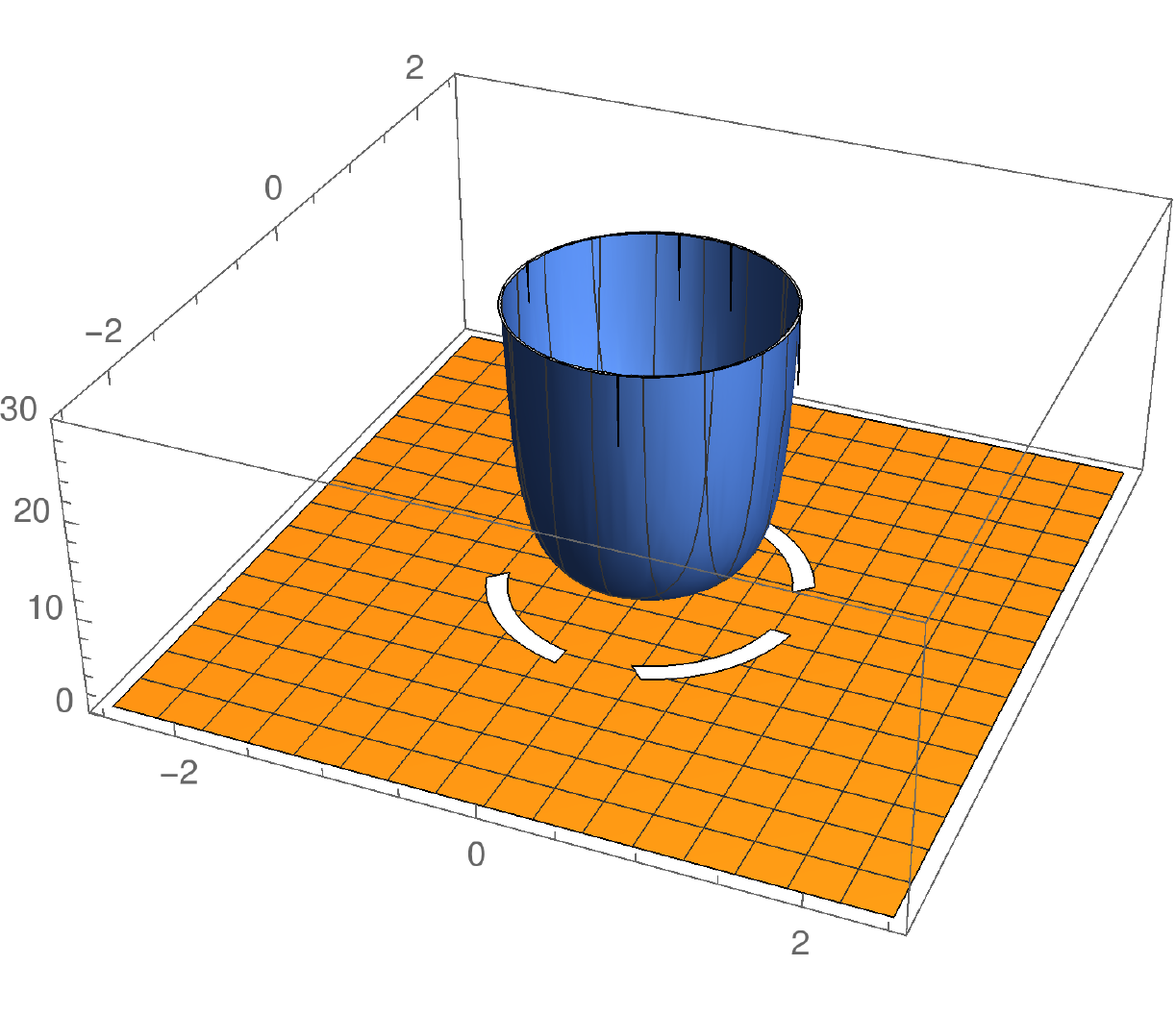}
\end{center}
\caption{Example of a solution of \eqref{p:ed} with $g\equiv 1$.}\label{5f}
\end{figure}

\begin{thm}[Theorem 1.4 in \cite{AJS17a}]\label{Eden1:thm}
Let $g\in C(\partial B)$ and $u:\R^N\to\R$ be given by 
\begin{align*}
u(x)=\int_{\partial B}E_{s-1}(x,\theta)\ g(\theta)\ d\theta\qquad \text{for $x\in \R^N$}.
\end{align*}
Then, $u\in C^{\infty}(B)$, $\delta^{1-s}u\in C(\overline{B})$, and
\begin{align}\label{p:ed:2}
(-\Delta)^s u=0\ \ \text{ in }\ \ B,\quad u=0\ \ \text{ on } \mathbb R^N\backslash\overline{B},\quad \dn^{s-2} u=0\ \ \text{ on } \partial B,\quad \dn^{s-1} u= g\ \ \text{ on } \partial B.
\end{align}
\end{thm}

\begin{figure}[h!]
\begin{center}
\includegraphics[height=5cm]{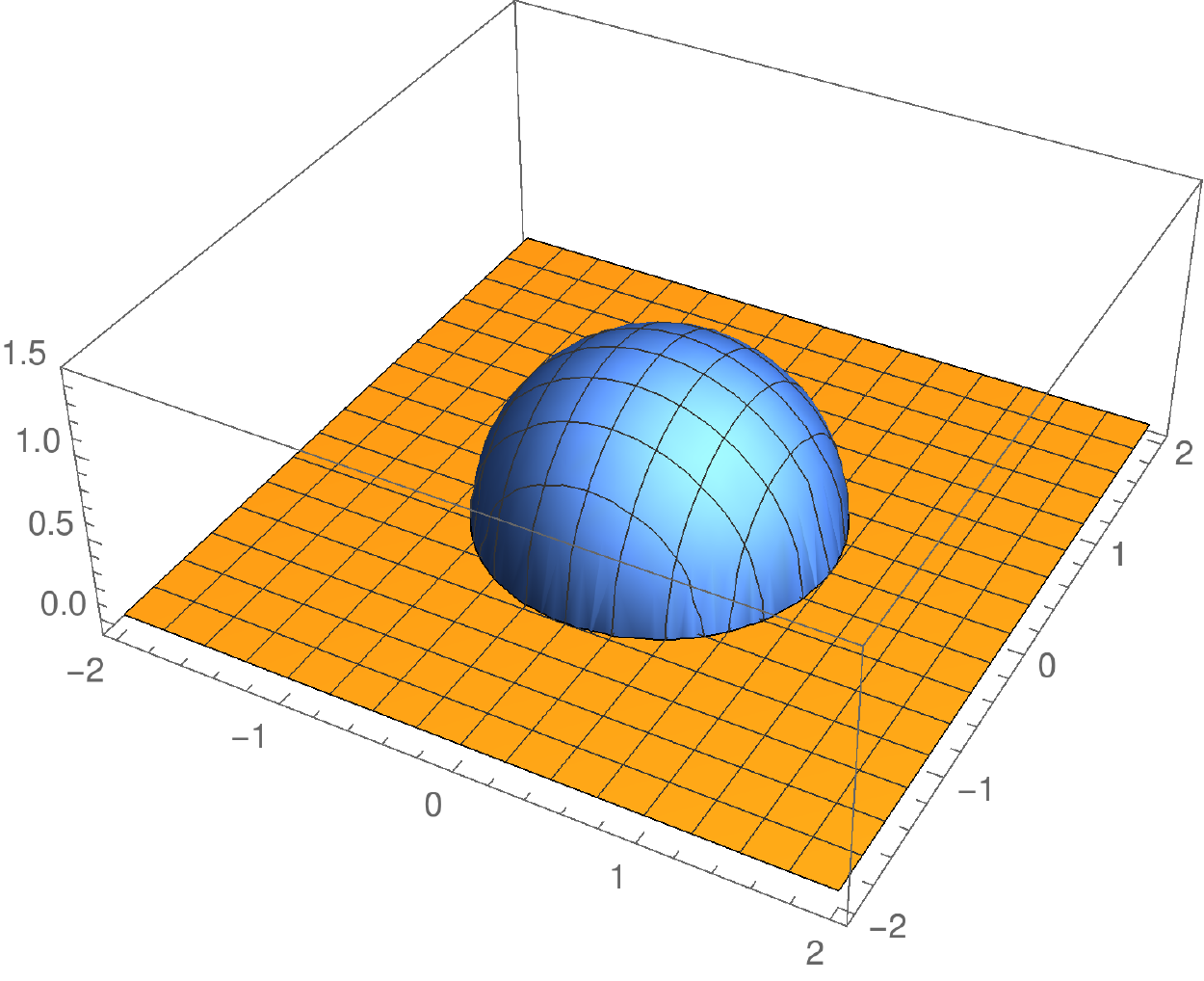}
\end{center}
\caption{Example of a solution of \eqref{p:ed:2} with $g\equiv 1$.}\label{6f}
\end{figure}

As mentioned at the beginning of the section, we see in Figure \ref{5f} that solutions of \eqref{p:ed} are \emph{singular} whenever $g\neq 0$ on $\partial B$.  In this sense, the trace $D^{s-2}u$ can also be seen as a \emph{singular trace} or \emph{singular profile} of $u$.   Figures \ref{3f}--\ref{6f} show the very different boundary behaviors that come into play when considering higher-order fractional boundary value problems. We also note that there is a relationship between the Green function $G_s$ and the boundary Poisson kernels, namely,
\begin{align*}
 E_{s-1}(x,z)&=c_1\lim_{y\to z}\frac{\partial^2}{\partial (|y|^2)^2}[\delta(y)^{2-s}G_s(x,y)],\quad 
 E_{s-2}(x,z)=c_2\lim_{y\to z}\frac{\partial^3}{\partial (|y|^2)^3}[\delta(y)^{2-s}G_{s}(x,y)]
\end{align*}
for $z\in\partial B$ and $x\in\R^N$, where $c_1$ and $c_2$ are suitable constants, see \cite[Lemma 1.8]{AJS17a}.

\subsection{The nonlocal Poisson kernel}

Let $m\in\{0,1\}$, $\sigma\in(0,1),$ and $s=m+\sigma\in(0,2)$.  For $x\in \mathbb R^N$ and $y\in \R^N\backslash \overline{B}$ let 
\begin{equation}\label{PK}
\Gamma_s(x,y)\ :=\ (-1)^m\frac{\gamma_{N,s-1}}{{|x-y|}^N}\frac{(1-|x|^2)_+^s}{(|y|^2-1)^s},
\end{equation}
where $\gamma_{N,s-1}$ is a positive normalization constant given in \eqref{c}.  The kernel \eqref{PK} is called the \emph{nonlocal Poisson kernel} for $(-\Delta)^s$ in $B$, and it can be used to construct $s$-harmonic solutions with prescribed values in $\R^N\backslash\overline{B}$.

\begin{thm}[Theorem 1.1 in \cite{AJS17a} and Theorem 1.6 in \cite{AJS17b}]\label{poisson:thm}
Let $s\in(1,2)$, $G_s$ as in \eqref{green}, and $\Gamma_s$ as in \eqref{PK}, then
\begin{align}\label{GsDsG}
\Gamma_s(x,y)=\ -(-\Delta)_y^s G_s(x,y)\qquad \text{for $x\in \mathbb R^N$, $y\in \R^N\setminus\overline{B}$} 
\end{align}
and, if $\psi\in \cL_s^1$ with $\psi=0$ in $B_{r}$ for some $r>1$,
and $u:\R^N\to \R$ is given by 
\begin{equation}\label{u:def}
u(x)\ =\int_{\R^N\setminus\overline{B}}\Gamma_s(x,y)\psi(y)\;dy\ +\ \psi(x),
 \end{equation}
then $u\in C^{\infty}(B)\cap C^s(B_{r})\cap H^s(B_{\rho})$ for any $\rho\in(1,r)$ and $u$ is the unique pointwise solution in the space $C^s(\overline{B})\cap H^s(B)$ of 
\begin{equation}\label{u:prob}
(-\Delta)^s u = 0\quad \text{ in }B,\qquad u=\psi\quad \text{ on }\R^N\backslash\overline{B},
\qquad D^{s-2}u=D^{s-1}u=0\quad \text{ on }\partial B.
\end{equation}
\end{thm}

\begin{figure}[h!]
\begin{center}
\includegraphics[height=5cm]{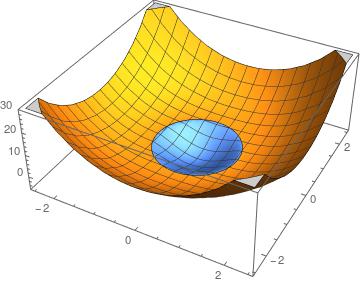}
\end{center}
\caption{Example of a solution of \eqref{u:prob}. Observe that although the data $\psi$ is positive in $\R^N\backslash B$, the solution $u$ is negative in $B$.  This is a higher-order phenomenon, which does not happen for $s \in(0,1)$, \cite{AJS17a}.
}
\end{figure}

Observe that in Theorem \ref{poisson:thm} the data $\psi$ is assumed to be zero close to the boundary $\partial B$, which implies that the traces \eqref{t:def} reduce to \eqref{red}.  If $\psi$ is not zero close to $\partial B$ then \eqref{u:def} might not be integrable and in this case one needs a different kernel. In fact, this case requires to use a lower-order nonlocal Poisson kernel together with a suitable correction using the boundary Poisson kernel $E_{s-1}$.  Recall that $\chi_U$ denotes the characteristic function of $U$, that is, $\chi_U(y)=1$ if $y\in U$ and $\chi_U(y)=0$ if $y\not\in U$. The next result is new.

\begin{thm}\label{xtra:thm} Let $\psi\in \cL_{s}^1\cap C^{s-1+\alpha}(B_r)$, $\psi_1:= \psi\chi_{B_r}$, $\psi_2:=\psi\chi_{\mathbb R^N\backslash B_r}$ for some $r>1$, $\alpha\in(0,1)$, and assume that
\begin{align}\label{hyp}
 z\mapsto \int_{\mathbb R^N\backslash \overline{B}}\frac{\psi_1(y)-\psi_1(z)}{|z-y|^N(|y|^2-1)^{s-1}}\ dy\quad \text{ belongs to }C(\partial B).
\end{align}
Let $u:\R^N\to \R$ be given by 
\begin{align*}\label{u:def}
u(x)\ =\int_{\R^N\setminus\overline{B}}&\Gamma_{s}(x,y)\psi_2(y)\;dy+\int_{\R^N\setminus\overline{B}}\Gamma_{s-1}(x,y)\psi_1(y)\;dy\\
&-2\gamma_{N,s-1}\int_{\partial B}E_{s-1}(x,z)\int_{\mathbb R^N\backslash \overline{B}}\frac{\psi_1(y)-\psi_1(z)}{|z-y|^N(|y|^2-1)^{s-1}}\ dy\ dz\ +\ \psi(x)\chi_{\R^N\backslash B},
\end{align*}
then $u\in C^{\infty}(B)\cap C^{s-1}(\overline{B})\cap \cL^1_s$ is a pointwise solution of
\begin{equation*}
(-\Delta)^s u = 0\quad \text{ in }B,\qquad u=\psi\quad \text{ on }\R^N\backslash\overline{B},
\qquad D^{s-2}u=D^{s-1}u=0\quad \text{ on }\partial B.
\end{equation*}
\end{thm}

We remark that, if $\psi\in \cL^1_{s-1}$, then the claim in Theorem \ref{xtra:thm} holds with $\psi_1\equiv \psi$ and $\psi_2\equiv 0$.  Property \eqref{hyp} is needed to use Theorem \ref{Eden1:thm} and its verification usually involves long computations but it can be easily verified in some simple situations, for example, if $\psi=c$ in $B_r$ for some $r>1$ and $c\in\R$, since in this case
\begin{align*}
 \int_{\mathbb R^N\backslash \overline{B}}\frac{\psi_1(y)-\psi_1(z)}{|z-y|^N(|y|^2-1)^{s-1}}\ dy
 =\int_{\mathbb R^N\backslash \overline{B_r}}\frac{\psi_1(y)-\psi_1(z)}{|z-y|^N(|y|^2-1)^{s-1}}\ dy
\end{align*}
and $|z-y|>r$ for all $z\in\partial B$ and $y\in\mathbb R^N\backslash \overline{B_r}$.  Observe also that, if $N=1$, then $\partial B=\{-1,1\}$ and \eqref{hyp} holds provided the integral is finite, which can be shown for any
\begin{align*}
\psi\in \cL^1_s\cap C^{s-1+\alpha,0}((-r,r))\qquad \text{ for some $r>1$ and $\alpha\in(0,1)$.} 
\end{align*}

\begin{proof}[Proof of Theorem \ref{xtra:thm}]
 Let $\psi,\psi_1,\psi_2,$ $r$, and $u$ be as in the statement. By \cite[Theorem 1.1]{AJS17a} we know that $v:\R^N\to\R$ given by
 \begin{align*}
  v(x)\ =\int_{\R^N\setminus\overline{B}}\Gamma_{s-1}(x,y)\psi_1(y)\;dy+\psi_1(x)\chi_{\R^N\backslash B}
 \end{align*}
belongs to $C^\infty(B)\cap \cL^1_{s-1}\cap C^{s-1}(B_r)$ and solves pointwisely $(-\Delta)^{s-1}v(x)=0$ in $B$.  Then, by Lemma \ref{eq:lem}, we also have that 
\begin{align}\label{e12}
\text{$(-\Delta)^{s}v(x)=(-\Delta)(-\Delta)^{s-1}v(x)=0$ in $B$. }
\end{align}
For $z\in\partial B$ let 
\begin{align}\label{long:comp}
{\varphi(z)}:={\gamma_{N,s-1}}\int_{\mathbb R^N\backslash\overline{B}}\frac{\psi_1(y)-\psi_1(z)}{(|y|^2-1)^{s-1}|z-y|^N}\ dy.
\end{align}
By \eqref{hyp} we have that $\varphi\in C(\partial B)$.  By Theorems \ref{Eden1:thm}, \ref{poisson:thm} (see also \cite[Theorem 1.1]{AJS17a}), and by \eqref{e12} it follows that $u\in C^{\infty}(B)\cap C^{s-1}(\overline{B})\cap \cL^1_s$ is a pointwise solution of $(-\Delta)^s u=0$ in $B$.  It remains to verify the boundary conditions. By definition, $u=\psi$ in $\R^N\backslash B$.  Moreover, since $u$ is bounded in $B_r$ for some $r>1$ we have that $\dn^{s-2} u=0$ on $\partial B.$ On the other hand, fix $z\in\partial B$ and recall that
\begin{align*}
 D^{s-1}u(z):=-\lim_{\substack {x\to z\\ x\in B}}\frac{\partial}{\partial|x|} \Big[\delta(x)^{2-s}\Big(u(x)
 -\lim_{\substack {y\to z\\ y\in \R^N\backslash B}}u(y)\Big)\Big].
\end{align*}
Using that the nonlocal Poisson kernel is normalized ($\int_{\mathbb R^N\backslash\overline{B}}\Gamma_{s-1}(\cdot,y)\ dy=1$ in $B$) we have that
\begin{align}\label{norm}
\psi_1(z)= \psi_1(z)\int_{\mathbb R^N\backslash\overline{B}}\Gamma_{s-1}(x,y)\ dy ={\gamma_{N,s-1}}\int_{\mathbb R^N\backslash\overline{B}}\frac{\delta(x)^{s-1}\psi_1(z)}{(|y|^2-1)^{s-1}|z-y|^N}\ dy\qquad \text{ for }x\in B.
\end{align}
Moreover, by \cite[Lemma 2.1]{AJS17a} we know that, if a function $f$ satisfies that $f=0$ in $\R^N\backslash \overline{B}$ and $D^{s-2}f=0$ on $\partial B$, then 
\begin{align}\label{p1}
 D^{s-1}f(z)=2\lim_{x\to z}\frac{f(x)}{\delta(x)^{s-1}}\qquad \text{ for }z\in\partial B.
\end{align}
For $x\in\R^N$, let $F(x):=\int_{\R^N\backslash\overline{B}}\Gamma_{s-1}(x,y) \psi_1(y)\ dy+\psi_1(x)\chi_{\R^N\backslash B}(x)$.  Observe that 
\begin{align*}
 \lim_{\substack {y\to z\\ y\in \R^N\backslash B}}F(y) 
 =\lim_{\substack {y\to z\\ y\in \R^N\backslash B}}\psi_1(y)\chi_{\R^N\backslash B}(y)
 = \psi_1(z)\qquad \text{ for }z\in\partial B
\end{align*}
and, by \eqref{norm}
\begin{align}\label{p2}
 F(x)-\lim_{\substack {w\to z\\ w\in \R^N\backslash B}}F(w)=
 {\gamma_{N,s-1}}\int_{\mathbb R^N\backslash\overline{B}}\frac{\delta(x)^{s-1}(\psi_1(y)-\psi_1(z))}{(|y|^2-1)^{s-1}|x-y|^N}\ dy\qquad \text{ for }x\in B.
\end{align}
Then, using \eqref{hyp}, \eqref{p1}, and \eqref{p2},
\begin{align*}
D^{s-1} F(z)=2{\gamma_{N,s-1}}\int_{\mathbb R^N\backslash\overline{B}}\frac{\psi_1(y)-\psi_1(z)}{(|y|^2-1)^{s-1}|z-y|^N}\ dy=2\varphi(z)\qquad \text{ for }z\in\partial B.
\end{align*}
But then, since 
\begin{align*}
&D^{s-1} \Bigg(\int_{\R^N\setminus\overline{B}}\Gamma_{s}(\cdot,y)\psi_2(y)\;dy+\psi_2\Bigg)=0\qquad \text{ on }\partial B,\\
&D^{s-1} \Bigg(-\int_{\partial B}E_{s-1}(\cdot,w){\varphi(w)}\ dw\Bigg)=-\varphi\qquad \text{ on }\partial B,
\end{align*}
by Theorems \ref{Eden1:thm} and \ref{poisson:thm}, we obtain that 
\begin{align*}
 D^{s-1} u
 =D^{s-1}\Bigg(\int_{\R^N\setminus\overline{B}}&\Gamma_{s}(\cdot,y)\psi_2(y)\;dy+F-2\int_{\partial B}E_{s-1}(\cdot,z)\varphi(y)\Bigg)
 =0\qquad \text{ on }\partial B,
\end{align*}
as claimed.
\end{proof}

\subsection{The Dirichlet boundary value problem in balls and representation formulas}

We can now put together all the previous results to obtain the following.

\begin{thm}[Theorem 1.4 in \cite{AJS17a} and Theorem 1.6 in \cite{AJS17b}]\label{main:thm}
Let $\alpha\in(0,1]$ such that $2s+\alpha\notin \N$, $g_k\in C^{1-k,0}(\partial B)$ for $k=0,1$, $f\in C^\alpha(\overline{B})$, $h\in \cL^1_{s}$ such that $h=0$ in $B_{r}$, $r>1$, and $u:\R^N\to\R$ be given by $u(x)=h(x)$ for $x\in\R^N\backslash\overline{B}$ and 
\begin{align*}
u(x)=\int_B G_s(x,y)f(y)\ dy+\int_{\R^N\backslash\overline{B}}\Gamma_s(x,y)h(y)\ dy+
\sum_{k=0}^{1}\ \int_{\partial B}E_{s-2+k}(x,\theta)\ g_k(\theta)\ d\theta\qquad \text{for $x\in B$}.
\end{align*}
Then, $u\in C^{2s+\alpha}(B)$, $\delta^{2-s}u\in C^{1,0}(\overline{B})$ and
\begin{align*}
(-\Delta)^s u=f\ \ \text{ in } B,\quad u=h\ \ \text{ on } \mathbb R^N\backslash\overline{B},\quad \dn^{s-2} u= g_0\ \ \text{ on } \partial B,\quad \dn^{s-1} u= g_1\ \ \text{ on } \partial B.
\end{align*}
\end{thm}

This solution given by Theorem \ref{main:thm} is, in fact, unique. As a consequence, we have the following representation formula.
\begin{thm}[Theorem 1.5 in \cite{AJS17a}]\label{uniqueness:thm}
Let $\alpha\in(0,1)$ such that $2s+\alpha\not\in \N$, $r>1$, $u\in \cL_{s}^1\cap C^{2s+\alpha}(B)$ be such that 
\begin{align*}
\delta^{2-s} u \in C^{1+\alpha}(\overline{B}),\qquad (-\Delta)^s u\in C^\alpha(\overline{B}),\qquad \text{ and }\qquad u=0\ \ \text{ in }B_{r}\backslash\overline{B}.
\end{align*}
Then, for $x\in B$,
\begin{align*}
u(x)=\int_B G_s(x,y)(-\Delta)^s u(y)\ dy + \int_{\R^N\backslash\overline{B}}\Gamma_s(x,y)u(y)\ dy+ \sum_{k=0}^{1}\ \int_{\partial B}E_{s-2+k}(x,\theta)\dn^{k+s-2} u(\theta)\ d\theta.
\end{align*}
\end{thm}

For a uniqueness statement in the case where $u$ is nonzero in $B_r\backslash \overline{B}$, we refer to \cite[Theorem 1.6]{AJS17a}.

\begin{remark}\label{sb2}
As mentioned in the introduction, similar results as those presented in this survey hold for $s>2$; however, one can also find surprising differences. To mention a few, let  $s=m+\sigma>2$ with $m\in \N$ and $\sigma\in(0,1)$.
\begin{enumerate}
\item A pointwise evaluation of $(-\Delta)^s u(x)$ can be calculated in terms of finite differences of order $2(m+1)$, namely,
\begin{align*}
	(-\Delta)^s u(x):=\frac{c_{N,\m+1,s}}{2}\int_{\R^N} \frac{\delta_{\m+1} u(x,y)}{|y|^{N+2s}} \ dy,
\end{align*}
where $\delta_{\m+1} u(x,y):= \sum_{k=-{\m+1}}^{\m+1} (-1)^k { \binom{2(m+1)}{m+1-k}} u(x+ky)$ and $c_{N,m+1,s}>0$ is a suitable normalization constant (see \cite[Equation (1.2)]{ADFJS18}).  In particular, direct calculations for the case $s>2$ are harder and they often require fine combinatorial identities.
\item The nonlocal Poisson kernel $\Gamma_s$ given in \eqref{PK} has the same formula for $s=m+\sigma>2$, namely, 
\begin{equation*}
	\Gamma_s(x,y)\ :=\ (-1)^m\frac{\gamma_{N,\sigma}}{{|x-y|}^N}\frac{(1-|x|^2)_+^s}{(|y|^2-1)^s}.
\end{equation*}
Observe that the Poisson kernel is \emph{positive} if $m$ is even and \emph{negative} if $m$ is odd.  As a consequence, the counterexample from Theorem \ref{main:thm:point} can be adjusted to show that $(-\Delta)^s$ does not have a general positivity preserving property if $s\in(m,m+1)$ with $m$ odd.  Remarkably, if $s\in(m,m+1)$ with $m$ even, this counterexample does not work anymore, and in fact, one can show that the Green's function for two disjoint balls is positive! (see \cite[Theorem 1.10]{AJS17a}).  Although maximum principles are not expected to hold for any $s>1$, a counterexample for $s\in(m,m+1)$ with $m\geq 2$ even, is still missing. 
\item  Problems with a nonlocal boundary condition $u=\psi$ in $\R^N\backslash B$ with $\psi\neq 0$ close to the boundary $\partial B$ are more intricate to characterize if $s>2$, and a result such as Theorem \ref{xtra:thm} is not yet available if $s>2$.
\item Similarly as in Section \ref{poisson:subsec}, for $s=m+\sigma>2$ one has $m+1$ boundary Poisson kernels.  Together with the nonlocal boundary condition (b.c.), we have that linear problems involving $(-\Delta)^s$ require $m+2$ b.c. to be well-posed. Note that $(-\Delta)^m$ and $(-\Delta)^{m+1}$ need only $m$ and $m+1$ b.c. respectively to be well-posed. The ``extra'' boundary condition is due to the nonlocality of $(-\Delta)^s$. 
\end{enumerate}
\end{remark}

\section{Asymptotic behavior of solutions}\label{asym:sec}

The convergence of the kernels $G_s$, $E_{s-2}$, $E_{s-1}$, and the corresponding solutions as $s\to 2^-$ is well behaved, in the sense that the (pointwise) limits exist and the resulting function is also a solution.  This can be easily verified (for suitable data) in virtue of Theorem \ref{main:thm} (see also \cite[Theorem 1.4]{AJS17a}) and the dominated convergence theorem.

Nevertheless, the limit as $s\to 1^+$ is more delicate and may not always yield a meaningful solution. We show this with a simple example: let $N=1$, $\sigma\in(0,1)$, and $s=1+\sigma$; then
\begin{align*}
 u_s(x):=\int_{\partial B} E_{s-2}(x,y)\ dy=E_{s-2}(x,-1)+E_{s-2}(x,1)=(1-x^2)^{s-2}=(1-x^2)^{\sigma-1}
\end{align*}
is a solution (by Theorem \ref{main:thm}) of $(-\Delta)^s u_s = 0$ in $B$ satisfying $D^{\sigma-1} u_s(z) = 1$ and $D^{\sigma} u_s(z) = 0$ for $z\in\partial B$.
If $\sigma\to 0$ (\emph{i.e.}, if $s\to 1^+$), then $u_s(x)\to (1-x^2)^{-1}$, which is \emph{not} harmonic in $B$. Note that $u_1\not\in L^1(B)$ and that the extra boundary condition ($D^{\sigma-1} u_s(z) = 1$) required in the higher-order case ($s\in(1,2)$) is \emph{incompatible} with problems of lower order ($s=1$).  In conclusion, there are sequences of $s$-harmonic functions which converge pointwisely to a function which is not harmonic; in other words, $s$-harmonicity is not a property that is always preserved in the limit without additional assumptions.

 On the other hand, note that if $\sigma\to 1^-$ (\emph{i.e.}, if $s\to 2^-$), then $u_s(x)\to 1$ pointwisely for $x\in B$, which is, in fact, a solution of 
 \begin{align*}
 (-\Delta)^2 u_2 = 0\quad \text{ in $B$},\qquad D^{0} u_2(z) = u_2(z) = 1,\ \ D^{1} u_2(z) = \frac{1}{2}\partial_\nu 1 = 0\quad \text{ for }z\in\partial B.
 \end{align*}

Regarding the kernel $\Gamma_s$, observe that the solution $u$ given by \eqref{u:def} goes uniformly to 0 as $s$ approaches 1 or 2, due to the constant $\gamma_{N,s-1}$ given in \eqref{c} and the assumption that $\psi=0$ in $B_r$ with $r>1$. Finally, for the convergence of the Poisson kernel $\Gamma_\sigma$ to the Poisson kernel for the Laplacian as $\sigma\to 1^-$ see \cite[footnote on page 121]{L72}.

\appendix

\section{On the composition of Green functions}
\label{A}

Let $G_t$ be given by Boggio's formula \eqref{green}.  In this appendix we show that the functions $u,v,w:\R^N\to\R$ given by
\begin{align}
u(x)&:=\ \int_{B} G_{s}(x,y)\ dy,\nonumber\\
v(x)&:=\ \int_{B}G_{s-1}(x,z)\int_{B}G_{1}(z,y)\ dy\ dz,\nonumber\\
w(x)&:=\ \int_{B}G_1(x,z)\int_{B}G_{s-1}(z,y)\ dy\ dz.\label{w:def}
\end{align}
solve each a different problem.  In particular, this also illustrates (see \eqref{not} below) the fact that the operators $(-\Delta)$ and $(-\Delta)^{s-1}$ can only be interchanged whenever the involved functions are well defined, and this can be a very subtle issue. 

To fix ideas, we focus on the case $s=\frac{3}{2}$ and $N=1$.  A first important difference between $u$, $v$, and $w$ comes from their optimal regularity. Indeed, by the regularity properties of Boggio's formula (see Theorem \ref{green:thm} or \cite{AJS17a}), we know that
\begin{align*}
u\in C^\infty(B) \cap C_0^s(\overline{B}),\qquad 
v\in C^\infty(B) \cap C_0^{\frac{1}{2}}(\overline{B}),\qquad 
\text{ and }\qquad 
w\in C^\infty(B)\cap C^{1}_0(B).
\end{align*}
In particular, we see that $(-\Delta)^s$ can always be computed in $B$.  We now argue that $u$, $v$, and $w$ solve in fact very different problems. First, we note that, by Theorem \ref{green:thm}, $u$ is the unique solution of \eqref{mp} with $f=1$ in $B$. Next, by construction, we see that $(-\Delta)^s v = (-\Delta)(-\Delta)^{\frac{1}{2}} v=1$ in $(-1,1)$ satisfying the boundary conditions
\begin{align*}
v=0\quad \text{ in }\R\backslash \overline{B}\qquad \text{ and }\qquad \lim_{x\to z}(-\Delta)^\frac{1}{2}v(x)=0\quad \text{ for }z\in \partial B. 
\end{align*}
Using a (fractional) Hopf Lemma (see \cite[Corollary 1.9]{AJS17a}), one can show that $D^{s-1}v>0$ on $\partial B$ and therefore $v$ does not satisfy the boundary conditions in \eqref{mp}.  We now turn our attention to the function $w$ given by \eqref{w:def}.  By \cite[Table 1 and Table 2]{D12}, we know that
\begin{align}\label{dyda}
 (-\Delta)^\frac{1}{2}\delta(x)^\frac{1}{2}=1\quad \text{ in }B\qquad \text{ and }\qquad 
 (-\Delta)^\frac{1}{2}(x\delta(x)^\frac{1}{2})=2x\quad \text{ in }B.
\end{align}
Then, 
\begin{align*}
 w_1:\R\to \R\quad \text{given by}\quad w_1(x)=\int_B G_\frac{1}{2}(x,y)\ dy
\end{align*}
is a solution of 
\begin{align*}
 (-\Delta)^\frac{1}{2} w_1 = 1\quad \text{ in }B,\qquad w_1=0\quad \text{ in }\R\backslash B
\end{align*}
and, by uniqueness (Theorem \ref{riesz:thm}) and \eqref{dyda},
\begin{align}\label{w1:def}
w_1(x)=\int_B G_\frac{1}{2}(x,y)\ dy=\delta(x)^\frac{1}{2}\qquad \text{ for }x\in\R. 
\end{align}
Since $G_1$ is the Dirichlet Green function of $(-\Delta)$ in $B=(-1,1)$, we have that $w$ given by \eqref{w:def} is a solution of 
\begin{align*}
 (-\Delta)w=w_1=\delta^\frac{1}{2}\quad \text{ in }B,\qquad w=0\quad \text{ in }\R\backslash B
\end{align*}
and $w$ satisfies
\begin{align*}
 w=0\quad \text{ in }\R\backslash \overline{B},\qquad \lim_{x\to z}(-\Delta)w(x)=0\quad \text{ for }z\in \partial B.
\end{align*}
Moreover, by Lemma \ref{eq:lem}, we have that $(-\Delta)^{\frac{3}{2}}w(x)=-\Delta(-\Delta)^{\frac{1}{2}}w(x)$ for $x\in B$.  If we na\"{\i}vely interchange the Laplacians, we would have that
\begin{align}\label{not}
(-\Delta)^{\frac{3}{2}}w(x)=-\Delta(-\Delta)^{\frac{1}{2}}w(x)=(-\Delta)^{\frac{1}{2}}(-\Delta)w(x)=(-\Delta)^{\frac{1}{2}}w_1(x)=1\qquad \text{ for }x\in B. 
\end{align}
However, such an interchange is not possible, because $\Delta w$ is not well defined on $\partial B$.\footnote{As can be seen in the proof of Lemma \ref{la}, $w$ is not twice weakly differentiable on $\partial B$ because its derivative $w'$ has a jump discontinuity at $\partial B$. Since $(-\Delta)^\frac{1}{2}$ is a nonlocal operator, the function $\Delta w$ must be well defined in $\R^N$.} In fact, we can show the following result.

\begin{lemma}\label{la}
 If $w$ is given by \eqref{w:def}, then $(-\Delta)^{\frac{3}{2}} w \neq 1$ in $B=(-1,1)$.
\end{lemma}
\begin{proof}
Let $\sin^{-1}$ denote the arcsine function and $d$ and $'$ denote derivatives. By direct computation, we have that
\begin{align*}
 W(y)=\frac{1}{12} \left(-2 \sqrt{1-y^2} \left(y^2+2\right)-6 y \sin ^{-1}(y)+3 \pi \right)\chi_{(-1,1)}(y)
\end{align*}
is the unique solution of 
\begin{align}\label{eq}
 -d^2W=-W''=\delta^{\frac{1}{2}}\quad \text{ in }(0,1),\qquad W(1)=0=W'(0),
\end{align}
and therefore, by \eqref{w1:def},
\begin{align*}
W(x)=w(x)=\int_{B}G(x,z)\delta^{\frac{1}{2}}(z)\ dz\qquad \text{ for }x\in B. 
\end{align*}
Then
\begin{align*}
 w'(y)=\frac{1}{2} \left(-\sqrt{1-y^2}y-\sin ^{-1}(y)\right)\chi_{(-1,1)}(y)\qquad \text{ in }\R\backslash\{\pm 1\}
\end{align*}
and $w'$ has a jump discontinuity\footnote{Thus $w'$ is not weakly differentiable in $\R$ and one cannot compute $(-\Delta)^{\frac{1}{2}}(-\Delta)w=(-\Delta)^{\frac{1}{2}}(-d^2)w$.} at $\pm 1$. By \cite[Corollary 1.4]{AJS17b} and \cite[Proposition B.2]{AJS16b} we have that
\begin{align*}
(-\Delta)^{1+{\frac{1}{2}}}w=-d^2(-\Delta)^{{\frac{1}{2}}}w=-d(-\Delta)^{{\frac{1}{2}}}dw
=-d(-\Delta)^{{\frac{1}{2}}}w'\qquad \text{ in }(-1,1). 
\end{align*}
By linearity and using \eqref{dyda},
\begin{align*}
 (-\Delta)^{\frac{1}{2}} [\frac{1}{2} \left(-\delta(y)^{\frac{1}{2}}  y-\sin ^{-1}(y)\right)\chi_{(-1,1)}]
 =\frac{1}{2}\left(-2y-(-\Delta)^{\frac{1}{2}}[\sin ^{-1}(y)\chi_{(-1,1)}(y)]\right),
\end{align*}
therefore,
\begin{align*}
	(-\Delta)^{\frac{3}{2}} w(y)=-d(-\Delta)^{\frac{1}{2}} d w(y) = 1 + d(-\Delta)^{\frac{1}{2}}[\sin ^{-1}(y)\chi_{(-1,1)}(y)]\qquad \text{ for }y\in(-1,1).
\end{align*}
Let $\zeta(y)=\sin ^{-1}(y)\chi_{(-1,1)}(y)$. To finish the proof, we show that $d(-\Delta)^{\frac{1}{2}}\zeta \neq 0$ in $(-1,1)$. Assume by contradiction that $d (-\Delta)^{\frac{1}{2}}\zeta=0$ in $B$, then, $(-\Delta)^{\frac{1}{2}}\zeta=c$ in $B$ for some $c\in\R$.  Now, we reach a contradiction using a representation result
\footnote{More direct approaches are also possible, but they require lengthy computations.  For example, since $\zeta$ has a jump discontinuity at $\pm1$, one can show that $\lim_{x\to 1^-}(-\Delta)^\frac{1}{2}\zeta(x)=\infty$.  Or, using integration by parts and the mean value theorem (see \cite[Lemma B.1]{AJS16b}), one can even compute explicitly the value of $(-\Delta)^\frac{1}{2}\zeta(x)$ at any $x\in(-1,1)$, but to keep this survey short we do not include the details of these calculations.}. Since $\delta^{{\frac{1}{2}}}\zeta\in C^\frac{1}{2}([-1,1])$, $\zeta\in L^\infty([-1,1])$, and $(-\Delta)^{\frac{1}{2}} \zeta\in C^\frac{1}{2}([-1,1])$, (actually, $(-\Delta)^{\frac{1}{2}} \zeta=c\in C^\infty([-1,1])$), we can use the representation of solutions \cite[Theorem 1.5]{AJS17a} and obtain that
\begin{align*}
 \sin ^{-1}(y)\chi_{(-1,1)}(y)=\zeta(y)=\int_BG_{\frac{1}{2}}(x,y) c\ dy = c(1-|y|^2)_+^{\frac{1}{2}}\qquad \text{ in }\R,
\end{align*}
a contradiction.  Therefore $d(-\Delta)^{\frac{1}{2}}[\sin ^{-1}(y)\chi_{(-1,1)}]\neq 0$ in $B$ and $(-\Delta)^sw\neq 1$ in $B$. 
\end{proof}

\subsection*{Acknowledgements}

I thank my coauthors and friends Nicola Abatangelo and Sven Jarohs for their comments and suggestions on how to improve this survey. 


\end{document}